\documentclass{amsart}

\usepackage[T1]{fontenc}
\usepackage[usenames,dvipsnames]{xcolor}
\usepackage{graphicx}
\usepackage{ifthen}
\usepackage{subcaption}
\usepackage{lineno}
\usepackage{mathrsfs}
\usepackage{ulem}
\usepackage[shortlabels]{enumitem}

\usepackage{hyperref}
\hypersetup{
    colorlinks=true,
    linktoc=all,
    linkcolor=blue
}

\usepackage{tikz}
\usetikzlibrary{decorations.pathmorphing}

\usepackage{textcomp}
\usepackage{amsfonts}
\usepackage{amssymb}
\usepackage{amsthm}
\usepackage{stmaryrd}

\newcommand{\reals}{\mathbb{R}}

\newcommand{\integers}{\mathbb{Z}}

\newtheorem*{corollary*}{Corollary}
\newtheorem*{theorem*}{Theorem}
\newtheorem{theorem}{Theorem}[section]
\newtheorem{lemma}[theorem]{Lemma}

\newtheorem{corollary}[theorem]{Corollary}
\newtheorem{observation}[theorem]{Observation}

\newtheorem*{claim*}{Claim}

\theoremstyle{definition}
\newtheorem{definition}[theorem]{Definition}

\newtheorem*{example*}{Example}
\newtheorem*{convention*}{Convention}

\theoremstyle{remark}

\newboolean{showcomments}
\setboolean{showcomments}{true}

\newcommand{\kjcomment}[1]%
    {\ifthenelse{\boolean{showcomments}}%
        {\textcolor{blue}{(N.B:[KJ] #1)}}{}}
\newcommand{\gkcomment}[1]%
    {\ifthenelse{\boolean{showcomments}}%
        {\textcolor{ForestGreen}{(N.B:[GK] #1)}}{}}
\newcommand{\kjcommentm}[1]%
    {\ifthenelse{\boolean{showcomments}}%
        {\marginpar{\textcolor{blue}{(N.B:[KJ] #1)}}}{}}
\newcommand{\gkcommentm}[1]%
    {\ifthenelse{\boolean{showcomments}}%
        {\marginpar{\textcolor{ForestGreen}{(N.B:[GK] #1)}}}{}}
\newcommand{\comment}[1]{\ifthenelse{\boolean{showcomments}}{\textcolor{red}{(N.B: #1)}}{}}

\begin{document}

\renewcommand{\bf}{\bfseries}
\renewcommand{\sc}{\scshape}
%insert defs/styles
\vspace{0.5in}

\title{On the asymmetry of stars at infinity}

\author{Keith Jones}
\address{Department of Mathematics, Computer Science \& Statistics; State University of New York College at Oneonta; 108 Ravine Parkway; Oneonta, NY 13820}
\email{keith.jones@oneonta.edu}

\author{Gregory A. Kelsey}
\address{Department of Mathematics; Bellarmine University; 2001 Newburg Rd.; Louisville, KY 40205}
\email{gkelsey@bellarmine.edu}

\subjclass[2010]{Primary 20F65, 20F69; Secondary 20E22, 05C25}

\keywords{stars at infinity, horofunction, horoboundary, Diestel-Leader graphs, lamplighter groups}

\thanks{The authors would like to thank the Institute for Advanced Study for
its hospitality at the Park City Math Institute Summer Session 2012, where
began conversations ultimately leading to this work. }

%\linenumbers

\begin{abstract} Given a bordified space, Karlsson defines an incidence
geometry of stars at infinity.  These stars and their incidence are closely
related to well-understood objects when the space is hyperbolic, CAT(0), or a
bounded convex domain with the Hilbert metric. A question stemming from
Karlsson's original paper was whether or not the relation of one boundary point
being included in a star of another boundary point is symmetric.
This paper provides an example demonstrating that this relation in the star
boundary of the three-tree Diestel-Leader graph $DL_3(q)$ is
not symmetric. In doing so, some interesting bounds on distance in
Diestel-Leader graphs are utilized.
\end{abstract}

\maketitle

\section{Introduction}

In \cite{karlssondynamics}, Karlsson presents a theory on the dynamics of
isometries and semicontractions of metric spaces in which he develops and
utilizes the idea of ``stars at infinity'' around boundary points of
bordified metric spaces, which essentially extend the notion of half-space to the
boundary of the space.
For example in a CAT(0) space, the star of a boundary point is the closed ball
of radius $\pi/2$ in the angular metric.
In a hyperbolic space, stars are singleton boundary points.

For a metric space $X$ with boundary $\partial X$, the star of a boundary
point $\eta$ is denoted $S(\eta)$.  Karlsson notes in Section 2.1 of
\cite{karlssondynamics} that, for $\eta, \xi \in \partial X$, it is unclear
whether, or under what conditions, $\xi \in S(\eta)$ implies $\eta \in S(\xi)$;
i.e., whether the relation of being included in the star is symmetric.  In this
paper, we exhibit an example in which this relation is not symmetric by
studying the horofunction boundary of the Diestel-Leader graph $DL_3(q)$ (to be
introduced in \S \ref{dlgraph:sec}), which is a Cayley graph of a kind of
generalization of the lamplighter group $L_2$. It should be noted that this
example lives outside the context of non-positively curved spaces. 
%TODO remove that last sentence?

\section{Stars at Infinity}
\label{boundary:sec}

\subsection{Background}

Karlsson introduces the following ideas in \cite{karlssondynamics}. Let $(X,x_0)$ be a based metric space. 
\begin{definition}
The {\it halfspace} for any $W\subset X$ with constant $C\geq0$ is given by:
$$H(W,C) = \{z\ |\ d(z,W)\leq d(z,x_0)+C\}.$$
For any bordification $\partial X$ of $X$, and any $\xi \in \partial X$ with
neighborhood basis $\mathscr U$, the {\it star} of $\xi$ is given by:
$$S(\xi) = \overline{\bigcup_{C\geq0}\bigcap_{U \in \mathscr
U}\overline{H(U,C)}}.$$
This is independent of choice of basepoint for $X$ and of neighborhood basis
for $\xi$. One can also consider the star of $\xi$ based at $x_0$, defined by:
\[ S^{x_0}(\xi) = \bigcap_{U \in \mathscr U} \overline{H(U, 0)}, \]
and note that $S^{x_0}(\xi) \subseteq S(\xi)$. 
\end{definition}

Bridson \& Haefliger provide an introduction to the horofunction boundary of a
metric space in \cite{bh}, 8.12. We provide a brief overview here.
Any based metric space $(X,x_0)$ has a natural embedding into the space
$\mathscr{C}_0(X)$ of continuous functions $X \rightarrow \reals$ with $f(x_0)
= 0$, via the mapping $x \in X \mapsto f_x(z) = d(x,z) - d(x,x_0)$. 
We give $\mathscr{C}_0(X)$ the compact-open topology and consider the closure
$\overline{X}$ of $X$ in this space. This closure is compact when $X$ is
proper, as is the space $\partial X = \overline{X} \backslash X$, which we call
the {\it horofunction boundary} of $X$. 

\begin{definition}
The {\it horofunction} $f: X\rightarrow \reals$ defined by a sequence $(x_n)$
is given by:
$$f(z) = \lim d(x_n, z) - d(x_n, x_0).$$
\end{definition}

If $(x_n)$ lies along a geodesic ray with $d(x_n,x_0) = n$, we call the induced 
horofunction a {\it Busemann function}.

\subsection{A lemma about star-inclusion} 
We make the following observation about stars.

\begin{lemma}
\label{seqstarinc:lem}
Let $X$ be a bordified metric space with basepoint $x_0$ and boundary $\partial X$.
Let $(x_n)$ and $(y_n)$ be sequences approaching points $\bar x$ and $\bar y$,
respectively, in $\partial X$. If for each $n$, $d(x_n,y_n) \leq
d(x_n,x_0)$, then $\bar x \in S^{x_0}(\bar y) \subseteq S(\bar y)$.
\end{lemma}

\begin{proof} Let $\{N_k\}$ be any neighborhood basis about
$\bar y$. Fix $k$. Then, since $y_n \rightarrow \bar y$, there exists a subsequence
$(s_n)$ of $(y_n)$ contained entirely in $N_k$, and $s_n \rightarrow \bar y$. Let
$(t_n)$ be the corresponding subsequence of $(x_n)$ (i.e., matching indices with
$s_n$), so that for each $n$,
$d(t_n, s_n) \leq d(t_n, x_0)$ and $t_n \rightarrow \bar x$. Then 
\[(s_n) \subseteq
N_k \implies (t_n) \subseteq H(N_k,0) \implies \bar x \in \overline{H(N_k,0)}.\]
Since $k$ was arbitrary, $\bar x \in S^{x_0}(\bar y)$. Recall, 
$S^{x_0}(\bar y) \subseteq S(\bar y)$.
\end{proof}

%There should be a generalized version of this Lemma that ensures $\bar x \in
%S(\bar y)$ if not $S^{x_0}(\bar y)$, but we will not need it.

\section{The Diestel-Leader Graph}
\label{dlgraph:sec}

\subsection{Background}

%Todo - do we want to use different letters for height functions vs
%horofunctions? Last discussion: probably not.

% Todo - do we want to define m and l before this definition and use that for
% calculating h? PROBABLY!  (Done)

% Todo - NEED TO DISCUSS EDGE LABELS!!! (Done)

% TODO - Say more about what Diestel-Leader graphs are like visually (maybe
% include a figure?) / give geometric intuition, prior to the Lemma on
% m(x,y),l(x,y).  (See third paragraph of definition. Will this be enough?)

\begin{definition}[The graph $DL_d(q)$] Let $T$ be a regular $q+1$ valent tree,
such that each vertex $v$ has a single predecessor and $q$ successors.  
We think of successors as lying above predecessors. Let
each edge have length 1, and label the edges of $T$ so that for each vertex
$v$, the $q$ successors of $v$ have labels in one-to-one correspondence with
the set $\{0,1,...,q-1\}$.  Choose a basepoint $o$ in $T$. For $v,w \in T$, let
$v \curlywedge w$ denote the greatest common ancestor of $v$ and $w$ in $T$.
Define the following functions $T \rightarrow Z$: 
\[ l(v) = d(v,o\curlywedge v), \quad 
m(v) = d(o,o\curlywedge v), \text{ and } \quad 
h(v) = l(v) - m(v).\]
The function $h$ gives the height in $T$, but we will make heavy use of $m$ and
$l$ as well, as they appear in the distance formula provided by Stein and Taback in
\cite{steintaback}.

For a positive integer $d$, let
$\{T_i \mid 1 \leq i \leq d\}$ be a set of copies of $T$ with basepoints
$o_i$ and functions $m_i$, $l_i$, and $h_i$. Let 
$DL_d(q)$ be the graph whose vertices are the $d$-tuples $v = (v_1,v_2,...,v_d)$,
$v_i \in T_i$, satisfying $\sum_{i=1}^d h_i(v_i) = 0$. Two vertices $v$ and $w$
in $DL_d(q)$ are joined by an edge if there are $i\neq j$ such that: (i)
$v_i$ and $w_i$ are adjacent in $T_i$, (ii) $v_j$ and $w_j$ are adjacent in
$T_j$, and (iii) for all $k \not \in \{i,j\}$, $v_k = w_k$. That is, two
vertices in $DL_d(q)$ are adjacent if you can get from one to the other by
simultaneously moving up in one tree and down in another. The graph $DL_d(q)$
has basepoint $o = (o_1,o_2,...,o_d)$; since we are interested in
cases where $DL_d(q)$ is the Cayley graph of a group, we refer to $o$ as $id$. 
There are natural projections $p_i: DL_d(q) \rightarrow T_i$ sending $v$ to
$v_i$. From here out, we will use $p_i(v)$ in lieu of $v_i$. We will reserve
the notation $d(v,w)$ for distance between two vertices in $DL_d(q)$, and we
will use $d_i(p_i(v),p_i(w))$ to refer to the distance from the projection
$p_i(v)$ to $p_i(w)$ in $T_i$.
\end{definition}

Notice that for $v \in DL_d(q)$, since $\sum h_i(v) = 0$, we have $\sum l_i(v)
= \sum m_i(v)$; each point $v$ is determined doing the following for each tree
$T_i$: first select the value $m_i(v)$, which represents moving downward in
$T_i$ to the height $-m_i(v)$, and then select a path upwards from that point
that does not backtrack having length $l_i(v)$. This upward path corresponds 
to an ordered tuple in
$\{0,1,...,q-1\}^{l_i(v)}$. Figure \ref{typicalpoint:fig} illustrates an example
element of $DL_3(2)$.

\begin{figure}
\begin{center}

\tikzset{
    vertex/.style = {shape=circle,fill=black,draw=black,inner sep=0pt,outer
sep=0pt,minimum size=1mm},
    %blank/.style = {minimum size=0mm},
    %edgy/.style = {decorate, decoration={snake, pre length=2mm, post length=2mm,
%amplitude=.3mm}}
}
\def\edgescale{.75cm}
\def\treeshift{3cm}
\begin{tikzpicture}

\draw [thin, gray, dashed] (-1cm,0) node[left] {$h=0$} -- (3*\treeshift,0);

\node[vertex] (A1) at ({.5*\treeshift+0*\edgescale},{0*\edgescale}) {};
\node[vertex] (B1) at ({.5*\treeshift-.5*\edgescale},{1*\edgescale}) {};
%\node[vertex] (C1) at ({1*\edgescale},{0*\edgescale}) {};
%\node[vertex] (D1) at ({1*\edgescale},{-2*\edgescale}) {};
%\node[vertex] (E1) at ({1.5*\edgescale},{-1*\edgescale}) {};

\draw (A1) node [fill=white, rectangle, left=.1cm] {$o_1$};
\draw (B1) node [above right] {$p_1(v)$}; 
%\draw (B) node [left] {$o_j \curlywedge z_j$};
%\draw (D) node [left] {$o_j \curlywedge x_{n,j}$};
%\draw (E1) node [left] {$y$};

\draw (A1) -- node [left] {\scriptsize 0} (B1);
%\draw (B1) -- (C1);
%\draw (B1) -- (D1);
%\draw (D1) -- (E1);

\node[vertex] (A2) at ({1*\treeshift+0*\edgescale},{0*\edgescale}) {};
\node[vertex] (B2) at ({1*\treeshift+.5*\edgescale},{-1*\edgescale}) {};
\node[vertex] (C2) at ({1*\treeshift+1*\edgescale},{0*\edgescale}) {};
\node[vertex] (D2) at ({1*\treeshift+1*\edgescale},{-2*\edgescale}) {};
\node[vertex] (E2) at ({1*\treeshift+1.5*\edgescale},{-1*\edgescale}) {};
%\node[vertex] (F2) at ({1*\treeshift+2*\edgescale}, 0) {};

\draw (A2) node [fill=white, rectangle, left=.1cm] {$o_2$};
\draw (C2) node [above right] {$p_2(v)$}; 
%\draw (D) node [left] {$\scriptsize{o_j \curlywedge z_j = o_j \curlywedge x_{n,j}}$};
%\draw (E) node [right] {$z_j \curlywedge x_{n,j}$};
%\draw (F2) node [left] {$y$};

\draw  (A2) -- node[left]{\scriptsize 0}(B2);
\draw (B2) -- node[left] {\scriptsize 0} (D2);
\draw (D2) -- node[right] {\scriptsize 1} (E2);
\draw (E2) -- node[left] {\scriptsize 0} (C2);
%\draw (E2) -- (F2);

\node[vertex] (A3) at ({2*\treeshift+0*\edgescale},{0*\edgescale}) {};
\node[vertex] (B3) at ({2*\treeshift+.5*\edgescale},{-1*\edgescale}) {};
%\node[vertex] (C3) at ({2*\treeshift+1*\edgescale},{0*\edgescale}) {};
\node[vertex] (D3) at ({2*\treeshift+1*\edgescale},{-2*\edgescale}) {};
\node[vertex] (E3) at ({2*\treeshift+1.5*\edgescale},{-1*\edgescale}) {};
%\node[vertex] (F3) at ({2*\treeshift+2*\edgescale}, 0) {};

\draw (A3) node [fill=white, rectangle, left=.1cm] {$o_3$};
%\draw (B) node [left] {$o_j \curlywedge x_{n,j}$};
%\draw (C3) node [left] {$y$}; 
%\draw (D) node [left] {$o_j \curlywedge z_j $};
%\draw (F3) node [left] {$x$};
\draw (E3) node [above right] {$p_3(v)$};

\draw (A3) -- node[left]{\scriptsize 0} (B3);
%\draw (B3) -- (C3);
\draw (B3) -- node[left] {\scriptsize 0} (D3);
\draw (D3) -- node[right] {\scriptsize 1} (E3);
%\draw (E3) -- (F3);
\end{tikzpicture}

\end{center}
\caption{A point $v \in DL_3(2)$ having:
$m_1 = 0$, $l_1=1$, $m_2 = l_2 = 2$, $m_3 = 2, l_3 =1$.
} 
\label{typicalpoint:fig}
\end{figure}
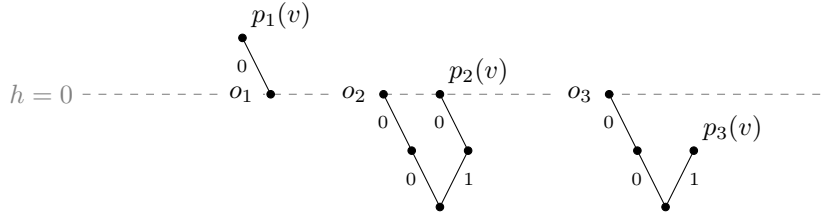

The graph $DL_d(q)$ is a special case of a more general graph, $DL(q_1,
q_2,...,q_d)$ built from $d$ trees having possibly different valences; all
of these are called Diestel-Leader graphs after the construction in
\cite{diestelleader} of an example of a vertex-symmetric graph that they
conjectured (in response to the question by Woess) is not 
quasi-isometric to the Cayley graph of any group. Eskin, Fisher, and Whyte
later proved in \cite{eskinfisherwhyte} that when $m\neq n$, this is indeed
the case for $DL(m,n)$. In this paper, we will only discuss $DL_3(q)$, which in
Corollary 3.15 of \cite{horoprod} is shown to be a Cayley graph of a certain
affine matrix group over $\integers/3\integers$ with respect to a certain
finite generating set.  %TODO Be more specific here?

%TODO Make more prominent?
Note: Throughout this paper, we are really interested in the vertex set of
$DL_d(q)$, representing the corresponding group with the metric structure
provided by the edges of the graph. Thus we abuse notation and use $DL_d(q)$ to
denote the discrete group.

% TODO Explain exactly why m(x,y) gives distance coming from m(x). 

% TODO Greater emphasis that not all DL graphs are cayley graphs and that
% distance function only proved for cayley graphs of groups. 

% TODO Is it worth commenting that we don't not know no a geometric proof of
% the distance
% function, and that this would be interesting to have?

\begin{definition} \label{def:mlandh} Let $x$ and $y$ be vertices of
$\text{DL}_d(q)$.  For $1\leq i \leq d$, we extend the $m,l$ notation to define
$m_i(x,y) = d_i(p_i(x), p_i(x)\curlywedge p_i(y)$,
$l_i(x,y) = d_i(p_i(y), p_i(x)\curlywedge p_i(y)$, and
$h_i(x,y) = l_i(x,y) - m_i(x,y)$.
Notice that $m_i(y) = m_i(id,y)$ and $l_i(y) = l(id,y)$, and 
$h_i(x,y) = h_i(y)-h_i(x)$.
\end{definition}

\begin{lemma}
\label{mixy:lem}
The formulas for $m_i(x,y)$ and $l_i(x,y)$ are determined by whether
$m_i(x)$ is less than, equal to, or greater than $m_i(y)$, as follows:
\begin{align*}
    m_i(x) < m_i(y)\text{: } & m_i(x,y) = l_i(x) + (m_i(y) - m_i(x)) = m_i(y) + h_i(x) \\
                     & l_i(x,y) = l_i(y)  \\
    m_i(x) = m_i(y)\text{: } &\text{set } D_i \text{ to the length 
                                of the common upward path} \\
                     & m_i(x,y) = l_i(x) - D_i \\
                     & l_i(x,y) = l_i(y) - D_i \\
    m_i(x) > m_i(y)\text{: } & m_i(x,y) = l_i(x) \\
                     & l_i(x,y) = l_i(y) + (m_i(x) - m_i(y)) = m_i(x) + h_i(y)
\end{align*}
\end{lemma}
\begin{proof}
The schematics for each case are illustrated in Figure \ref{milicases:fig}. 
\end{proof}

\begin{figure}
\begin{center}

\tikzset{
    vertex/.style = {shape=circle,fill=black,draw=black,inner sep=0pt,outer
sep=0pt,minimum size=1mm},
    blank/.style = {minimum size=0mm},
    edgy/.style = {decorate, decoration={snake, pre length=2mm, post length=2mm,
amplitude=.3mm}}
}
\def\edgescale{1.5cm}

\begin{tikzpicture}

\node[vertex] (A) at ({0*\edgescale},{0*\edgescale}) {};
\node[vertex] (B) at ({.5*\edgescale},{-1*\edgescale}) {};
\node[vertex] (C) at ({1*\edgescale},{0*\edgescale}) {};
\node[vertex] (D) at ({1*\edgescale},{-2*\edgescale}) {};
\node[vertex] (E) at ({1.5*\edgescale},{-1*\edgescale}) {};

\draw (A) node [left] {$o_j$};
%\draw (B) node [left] {$o_j \curlywedge z_j$};
\draw (C) node [left] {$x$}; 
%\draw (D) node [left] {$o_j \curlywedge x_{n,j}$};
\draw (E) node [left] {$y$};

\draw [dashed] (A) -- (B);
\draw [dashed] (B) -- (C);
\draw [edgy] (B) -- (D);
\draw [edgy] (D) -- (E);

\end{tikzpicture}
\hspace{.1in}\vline\hspace{.1in}
\begin{tikzpicture}

\node[vertex] (A) at ({0*\edgescale},{0*\edgescale}) {};
\node[vertex] (B) at ({.5*\edgescale},{-1*\edgescale}) {};
\node[vertex] (C) at ({1*\edgescale},{0*\edgescale}) {};
\node[vertex] (D) at ({1*\edgescale},{-2*\edgescale}) {};
\node[vertex] (E) at ({1.5*\edgescale},{-1*\edgescale}) {};
\node[vertex] (F) at ({2*\edgescale}, 0) {};

\draw (A) node [left] {$o_j$};
\draw (C) node [left] {$x$}; 
%\draw (D) node [left] {$\scriptsize{o_j \curlywedge z_j = o_j \curlywedge x_{n,j}}$};
%\draw (E) node [right] {$z_j \curlywedge x_{n,j}$};
\draw (F) node [left] {$y$};

\draw [dashed] (A) -- (B);
\draw [dashed] (B) -- (D);
\draw [dashed] (D) -- (E);
\draw [dashed] (E) -- (C);
\draw [edgy] (E) -- (F);
\end{tikzpicture}
\hspace{.1in}\vline\hspace{.1in}
\begin{tikzpicture}

\node[vertex] (A) at ({0*\edgescale},{0*\edgescale}) {};
\node[vertex] (B) at ({.5*\edgescale},{-1*\edgescale}) {};
\node[vertex] (C) at ({1*\edgescale},{0*\edgescale}) {};
\node[vertex] (D) at ({1*\edgescale},{-2*\edgescale}) {};
\node[vertex] (E) at ({1.5*\edgescale},{-1*\edgescale}) {};
\node[vertex] (F) at ({2*\edgescale}, 0) {};

\draw (A) node [left] {$o_j$};
%\draw (B) node [left] {$o_j \curlywedge x_{n,j}$};
\draw (C) node [left] {$y$}; 
%\draw (D) node [left] {$o_j \curlywedge z_j $};
\draw (F) node [left] {$x$};

\draw [dashed] (A) -- (B);
\draw [edgy] (B) -- (C);
\draw [dashed] (B) -- (D);
\draw [dashed] (D) -- (E);
\draw [dashed] (E) -- (F);
\end{tikzpicture}

\end{center}
\caption{Schematics for the cases for Lemma \ref{mixy:lem}.} 
\label{milicases:fig}
\end{figure}
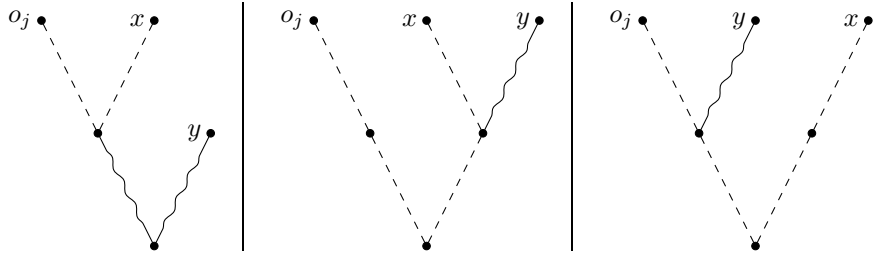

%\subsection{Distance in Diestel-Leader graphs.}

\begin{definition}
\label{mlformulas:def}
Let $x$ and $y$ be vertices in $\text{DL}_d(q)$ and let $\sigma\in\Sigma_d$, 
the symmetric group on $d$ letters.
For $2\leq i \leq d-1$, define
\[ f_{\sigma,i}(x,y) = 
m_{\sigma(1)}(x,y) + \dots +m_{\sigma(i)}(x,y) + l_{\sigma(i)}(x,y)+ \dots +
l_{\sigma(d)}(x,y), \]
and 
\[ f_{\sigma,d}(x,y) = 2m_{\sigma(1)}(x,y) + m_{\sigma(2)}(x,y) + \dots +
m_{\sigma(d)}(x,y) + l_{\sigma(d)}(x,y).\]
\end{definition}

Stein and Taback derive the following distance formula in \cite{steintaback}
(see Lemma 1 and following discussion, as well as the proof of Corollary 10)
%TODO keep this parenthetical?
in the case that $DL_d(q)$ is the Cayley graph of a group. 

\begin{theorem}[Stein-Taback]
\label{STdistthm}
Let $x$ and $y$ be vertices in $\text{DL}_d(q)$. For $\sigma \in \Sigma_d$, let
$f_\sigma(x,y) = \max_{2\leq i \leq d}\{ f_{\sigma,i}(x,y) \}$. Then
$d(x,y) = \min_{\sigma\in\Sigma_d}\left\{ f_{\sigma}(x,y) \right\}$.
\end{theorem}

\subsection{Distance bounds in Diestel-Leader graphs} 
One can find a variety of
lower bounds on the distance between two points in $DL_d(q)$. 

\begin{observation}
\label{bigindex:obs}
Let $R > 0$. Let $x, y \in DL_d(q)$, and that suppose for each
$\sigma \in \Sigma_d$, there is $i \in \{2, ..., d\}$ such that
$f_{\sigma,i}(x,y) \geq R$. Then $d(x,y) \geq R$.
\end{observation}

\begin{observation} 
\label{treedist:obs}
Let $x, y \in DL_d(q)$, and let $R = \max\{d_i(p_i(x),p_i(y)) \mid 1 \leq i \leq
d\}$. Then $d(x,y) \geq R$. 
\end{observation}

\begin{lemma}%[Consistent inequalities for $f_{\sigma,i}$ imply inequality of distance]
\label{f-ineq:lem}
Let $x,y,z \in DL_d(q)$ and $k \geq 0$. Suppose that for all $\sigma
\in \Sigma_d$ and all $i \in \{2, ..., d\}$, $f_{\sigma,i}(x,z) \geq f_{\sigma, i}(x,y) +k$. Then
$d(x,z) \geq d(x,y) + k$. The same is true when the inequalities are not strict.
\end{lemma} 

\begin{proof}
For each $\sigma \in \Sigma_d$, 
\[ f_\sigma(x,z) = \max_{2\leq i\leq d}\{ f_{\sigma,i}(x,z) \} \geq \max_{2\leq
i\leq d}\{ f_{\sigma,i}(x,y) \} + k = f_\sigma(x,y) + k. \]
So,
\[ d(x,z) = \min_{\sigma \in \Sigma_d}\{f_\sigma(x,z)\} \geq \min_{\sigma \in
\Sigma_d}\{f_\sigma(x,y)\} + k = d(x,y) + k. \]
%Because the distance function $d(x,z)$, for $x,z \in DL_3(q)$, is given by 
%\[ d(x,y) = \min_{\sigma \in \Sigma_d} \{ \max_{2 \leq i \leq d} \{
%f_{\sigma,i}(x,z)\} \}, \]
%and $f_{\sigma,i}(x,y) < f_{\sigma,i}(x,z)$ for each $\sigma$ and $i$, 
%this result is a consequence of Observation \ref{numsetsrelations:obs}.
\end{proof}

\begin{lemma}
\label{m-ineq:lem}
Let $x, y, z \in DL_d(q)$ and suppose there are nonnegative numbers $c_j$,
$1 \leq j \leq d$, such that $m_j(x,z) \geq m_j(x,y) + c_j$ and
$l_j(x,z) \geq l_j(x,y) + c_j$. Then $d(x,z) \geq d(x,y) + \sum_{j=1}^d c_j$.

%Moreover, if additionally there exists at least one index $j_0$ such that
%$m_{j_0}(x,y) < m_{j_0}(x,z)$ and $l_{j_0}(x,y) < l_{j_0}(x,z)$, then $d(x,y) < d(x,z)$.
\end{lemma}

% TODO rewrite this proof so that it is more clear.

\begin{proof} Investigating Definition \ref{mlformulas:def}, we see that 
for each $\sigma \in \Sigma_d$ and $2 \leq i \leq d$, 
$f_{\sigma,i}(x,z)$ is a sum that can be decomposed into $d$ terms, one for
each tree, such that for $1 \leq j \leq d$, tree $T_j$ contributes exactly one of:
$m_j(x,z)$, $2m_j(x,z)$, $l_j(x,z)$, or $m_j(x,z) + l_j(x,z)$.
Note that the term contributed depends only on $\sigma$ and $i$ and does not
depend on $x$ or $z$.

The assumption that $m_j(x,z) \geq m_j(x,y) + c_j$ and $l_j(x,z) \geq l_j(x,y)
+ c_j$ for each $j$ ensures that  
%\[g_j(x,z) \geq g_j(x,y) + \epsilon c_j \geq g_j(x,y) + c_j,\ \text{ where }
%\epsilon \in \{1,2\}.\]
$f_{\sigma,i}(x,z) \geq f_{\sigma,i}(x,y) + \sum_{j=1}^d c_j$. 
By Lemma \ref{f-ineq:lem}, $d(x,z) \geq d(x,y) + \sum_{j=1}^d c_j$.
%%%The conditions $m_j(x,y) \geq l_j(x,z)$ and $l_j(x,y) \geq l_j(x,z)$ ensure
%%%that for each $\sigma \in \Sigma_d$ and each $j \in \{2, ..., d\}$,
%%%$f_{\sigma,j}(x,y) \geq f_{\sigma,j}(x,z)$, since each $f_{\sigma,j}$ is
%%%defined as a sum of $m_j$ and $l_j$ terms. By Lemma \ref{f-ineq:lem},
%%%$d(x,y) \geq d(x,z)$.
%%%
%%% This ensures that if, in addition to the non-strict inequalities for
%%%all indices, there is a $j_0$ such that the inequality is strict, then
%%%$f_{\sigma,i}(x,y) < f_{\sigma,i}(x,z)$ for each $\sigma$ and $i$. By Lemma
%%%\ref{f-ineq:lem}, $d(x,y) < d(x,z)$.
\end{proof}

The next results deal with points $z$ that have $h_i(z) = 0$ for each $i$.
Such points are special because $h_i(z) = 0$ implies that $m_i(z) = l_i(z)$,
which can be useful in understanding distance. 
Also, since $id$ is such a point, it is easier to compare distances to these points
with distances to $id$.

\begin{lemma}[Cases where $d(x,z) \leq d(x,id)$]
\label{nonpos-horo:lem}
Let $x, z \in DL_d(q)$ and suppose $h_i(z) = 0$ for $1 \leq i \leq d$.
\begin{enumerate}[(i)]
\item \label{strict:part} If $m_i(z) < m_i(x)$ for each
$i$ such that $m_i(x) \neq 0$, and $m_i(z) = 0$ for
each $i$ such that $m_i(x) = 0$, then $d(x,z) = d(x,id)$.
\item \label{nonstrict:part} If $m_i(z) \leq m_i(x)$ for each $i$, then $d(x,z) \leq d(x,id)$.
\end{enumerate}
\end{lemma}

% TODO Do we want to make the wording here more clear?

\begin{proof} We apply Lemma \ref{mixy:lem}, where we establish the parameter
$D_i$ for path overlap in cases where $m_i(z) = m_i(x)$.

For part \ref{strict:part}, if $m_i(x) \neq 0$, $m_i(z) < m_i(x)$ implies
$m_i(z,x) = m_i(x) + h_i(z) = m_i(x)$ and $l_i(z,x) = l_i(x).$ 
If $m_i(x) = 0$, the path overlap $D_i =0$ since $l_i(z) = 0$, so that
$m_i(z,x) = l_i(z) = m_i(x)$ and $l_i(z,x) = l_i(x)$.
Thus for each $\sigma$ and $j$, $f_{\sigma,j}(z,x) = f_{\sigma,j}(id,x)$,
so that $d(x,z) = d(x,id)$.

For part \ref{nonstrict:part}, for $i$ such that $m_i(x) = m_i(z) > 0$,
we replace the equations from the previous paragraph with inequalities
$m_i(z,x) = l_i(z) - D_i = - D_i \leq m_i(x)$,
and 
$l_i(z,x) = l_i(x) - D_i \leq l_i(x).$ 
By Lemma \ref{m-ineq:lem}, $d(x,z) \leq d(x,id)$.
\end{proof}

\begin{lemma}[Cases where $d(x,z) \geq d(x,id) + k$]
\label{pos-horo:lem}
Let $x,z \in DL_d(q)$ and suppose $h_i(z) = 0$ for each $i$. Assume $m_i(z) \neq
m_i(x)$ except in cases where $m_i(x) = 0$. Let $c_i = \max\{0, m_i(z) - m_i(x)\}$. 
Then $d(x,z) \geq d(x,id) + \sum_{i=1}^d c_i$.
\end{lemma}
\begin{proof}
For each $i$, we apply Lemma \ref{mixy:lem} to consider $m_i(z,x)$ and
$l_i(x,z)$ in the following cases.

If $m_i(z) < m_i(x)$, then 
\[m_i(z,x) = m_i(x) + h_i(z) = m_i(x), \quad \text{ and } \quad l_i(z,x) =
l_i(x). \] 

If $m_i(z) = m_i(x) = 0$, then the path overlap $D_i=0$ since $l_i(z) = 0$,
so that 
\[ m_i(z,x) = l_i(z) - D_i = m_i(x) \quad \text{ and } \quad 
   l_i(z,x) = l_i(x) - D_i = l_i(x). \]

Finally, if $m_i(z) > m_i(x)$, then 
$m_i(z,x) = l_i(z) = m_i(z) = m_i(x) + c_i$
and
\[l_i(z,x) = m_i(z) + h_i(x) = (m_i(x) + c_i) + (l_i(x) - m_i(x)) = l_i(x) + c_i.
\]

By Lemma \ref{m-ineq:lem}, $d(x,z) \geq d(x,id) + \sum_{i=1}^d c_i$.
\end{proof}

\section{Horofunctions and Stars in $DL_d(q)$}

%%%\begin{definition} 
%%%\label{def:minfinite} We say a sequence $(x_n) \in DL_d(q)$
%%%is $m_i$-infinite, for $1\leq i\leq d$, if $\limsup m_i(x_n) = \infty$, and
%%%$m_i$-finite otherwise. 
%%%\end{definition}

% TODO Where do we use this statement? If we do not use it, ditch it. (DITCHED
% for now.)

%%%\begin{observation} 
%%%\label{minfnonempty:obs} 
%%%Since $\{z \in DL_d(q) \mid m_i(z) < M \}$ is finite for any $M > 0$, it follows
%%%that if a sequence $(x_n)$ approaches a point in
%%%$\partial_h(\text{DL}_d(q))$, then it is $m_i$-infinite for at least one $i$.
%%%\end{observation}

\subsection{$m$-invariance of horofunctions}
\begin{definition}[The point $\zeta^i_k$]
\label{zeta:def}
For $1 \leq i \leq d$ and $k>0$, let $\zeta^i_k \in DL_d(q)$ be the point given by
$m_i(\zeta^i_k) = l_i(\zeta^i_k) = k$, following upward edges labeled ``1'', and
$m_j(\zeta^i_k)= l_j(\zeta^i_k) = 0$ for $j \neq i$.
\end{definition}

\begin{lemma}%[$\lim m_i$ exists, for $1\leq i \leq d$, in a horofunction]
\label{mlim:lem}
Let $x_n$ be a sequence defining a horofunction $h_x$. Then for each $1\leq i
\leq d$, $\lim m_i(x_n)$ exists in $\integers_{\geq 0} \cup \{\infty\}$.
\end{lemma}

\begin{proof} Suppose for some $i$ that $\lim m_i(x_n)$ does not exist
in $\integers_{\geq 0} \cup \{\infty\}$. Then since $m_i$ takes nonnegative integer values,
there exist two subsequences $(y_n)$ and $(z_n)$ of $(x_n)$, both approaching
$h_x$, such that $m_i(y_n)$ is a constant $m_y$ and $m_i(z_n)$ and is bounded
below by $m_y + 1$. 

Let $k = m_y + 1$. Lemma
\ref{pos-horo:lem} gives $d(y_n,\zeta^i_k) > d(y_n,id)$ for each
$n$, implying $h_x(\zeta^i_k) > 0$. But by Lemma
\ref{nonpos-horo:lem}, $d(z_n, \zeta^i_k) \leq d(z_n, id)$ for each $n$,
implying $h_x(\zeta^i_k) \leq 0$, which is a contradiction.  \end{proof}

The proof of Lemma \ref{mlim:lem} also proves:
\begin{theorem}[$m_i$-Invariance]
\label{minvariance:thm}
If for sequences $(y_n)$ and $(z_n)$ defining horofunctions $h_y$ and $h_z$,
there is an index $i$ such that $\lim m_i(y_n) \neq \lim m_i(z_n)$, then $h_y
\neq h_z$.
\end{theorem}

In light of Theorem \ref{minvariance:thm}, for a horofunction $h$, we extend our
notation to let $m_i(h)=\lim m_i(x_n)$ for any $(x_n)$ approaching $h$.

\begin{corollary} \label{mconv:cor}
If a sequence $(h_n)$ of horofunctions converges to a
horofunction $h$, then $m_i(h_n)$ converges to $m_i(h)$ for each $i$.
\end{corollary} 

\begin{proof} For each $z \in DL_d(q)$, the sequence $(h_n(z))$ must become
constant on $z$ for large enough $n$, since the codomain, $\integers$, is
discrete. By $m$-invariance, as in the proof of Lemma \ref{mlim:lem}, if there
was any $i$ such that $m_i(h_n)$ did not also eventually become constant, there
would exist $z$ for which that did not occur.  \end{proof}

\subsection{Two Horofunctions}

We will demonstrate $\alpha, \beta \in \partial DL_3(q)$ with $\beta \in
S(\alpha)$ and $\alpha \not \in S(\beta)$.

Let $(\alpha_n)$ be the sequence such that $l_1(\alpha_n) = m_2(\alpha_n) = n$,
all other $m_i, l_i = 0$, and $(\alpha_n)$ moves upward in $T_1$ choosing edges
labeled ``1''. The function $\alpha(z) = \lim_{n\rightarrow \infty}
(d(\alpha_n, z) - n$ is a Busemann function, and therefore a horofunction
(\cite{bh}, 8.17-8.18).
%TODO what? this was just marked "TODO," anything needed here?
 While it is not difficult to calculate $\alpha$ in general, we will not need
to.

Let $(\beta_n)$ be the sequence $m_3(\beta_n) = l_3(\beta_n) = n$, $m_i(\beta) =
l_i(\beta) = 0$ for $i = 1,2$, and the path upward in $T_3$ always selects the
``1'' edge.  

% TODO Do we want to write this so that we have min_{j = 1,2} (m_j + h_3, ...)
% CHANGED; original is commented out.

\begin{lemma}
\label{hbeta:lem}
The sequence $(\beta_n)$ defines a horofunction $\beta$, and 
%\[\beta(z) = m_1 + m_2 + \min\{m_1 + h_3,\ m_2 + h_3,\ h_1 + h_3,\ h_2
%+ h_3,\ l_1,\ l_2\},\]
\[\beta(z) = m_1 + m_2 + \min_{j=1,2}\{m_j + h_3, h_j + h_3, l_j\}, \]
where each $m_i=m_i(z),l_i=l_i(z),h_i=h_i(z)$.
\end{lemma}

\begin{proof}
Let $z \in DL_3(q)$ and denote $m_i(z)$, $l_i(z)$, and $h_i(z)$ by
$m_i$, $l_i$, and $h_i$ for convenience.  Assume $n >> m_i,l_i$ for each $i$. 
We have $m_3(\beta_n,z) = n$,
$l_3(\beta_n,z) = n+h_3$, and $m_i(\beta_n,z) = m_i$ and
$l_i(\beta_n, z) = l_i$ for $i=1,2$. Figure
\ref{betandist:fig} shows the various $f_{\sigma}(\beta_n,z)$.
The only term all $f_{\sigma}$ have in common is $2n$, but we can rewrite
$l_i = m_i + h_i$, and after this translation all terms have an
additional $m_1 + m_2$ in common. This yields: 
\[ d(\beta_n,z) = 2n + m_1 + m_2 + \min\{m_1 + h_3,\ m_2 + h_3,\ l_1,\ h_2 +
h_3,\ h_1 + h_3,\ l_2\}. \]
%Since we can move from $\beta_n$ to $id$ in $2n$ steps by moving down $n$, then
%up $n$ steps $T_3$, compensating in other trees, Observation \ref{treedist:obs}
%ensures that $d(\beta_n,id) = d_3(p_3(\beta_n),o_3)= 2n$.  
This yields $d(\beta_n, id) = 2n$,  and $\beta(z) =
\lim_{n\rightarrow \infty} d(\beta_n,z) - 2n$.
\end{proof}

\begin{figure}
\begin{center}
\begin{tabular}{|c|c|c|c|c|}
\hline
$\sigma$ & $i$ & $f_{\sigma,i}$ & coeff. of $n$ & $f_{\sigma}$ \\
\hline
(1)  & 2 & $m_1 + m_2 + l_2 + (n + h_3)$ & 1  & \\ 
     & 3 & $2m_1 + m_2 + n + (n + h_3)$  & 2 & $2n + 2m_1 + m_2 + h_3$ \\
\hline
(12) & 2 & $m_2 + m_1 + l_1 + (n + h_3)$ & 1  & \\
     & 3 & $2m_2 + m_1 + n + (n + h_3)$  & 2 & $2n + m_1 + 2m_2 + h_3$ \\
\hline
(13) & 2 & $n + m_2 + l_2 + l_1$       & 1  & \\
     & 3 & $2n + m_2 + m_1 + l_1$      & 2 & $2n + m_1 + l_1 + m_2$ \\
\hline
(23) & 2 & $m_1 + n+ (n+h_3) + l_2$      & 2 & $2n + m_1 + l_2 + h_3$ \\
     & 3 & $2m_1 + n+ m_2 + l_2$       & 1  & \\
\hline
(123)& 2 & $m_2 + n + (n+h_3) + l_1$     & 2 & $2n + l_1 + m_2 + h_3$ \\
     & 3 & $2m_2 + n + m_1 + l_1$      & 1  & \\
\hline
(132)& 2 & $n + m_1 + l_1 + l_2$       & 1  & \\
     & 3 & $2n + m_1 + m_2 + l_2$      & 2 & $2n + m_1 + m_2 + l_2$  \\ 
\hline
\end{tabular}
\end{center}
\caption{Distance from $\beta_n$ to an arbitrary point $z$, when $n$ is very
large relative to the parameters of $z$.}
\label{betandist:fig}
\end{figure}

%%If $h_3 = 0$, we can further simplify this to
%%\[ d(\beta_n,a) = 2n + m_1 + m_2 + \min\{m_1,\ m_2,\ l_1,\ l_2,\ -h_1,\ -h_2\} \]

\subsection{Neighborhoods and Stars}
\begin{lemma} $\beta \in S(\alpha)$. \end{lemma}
\begin{proof}
As noted, $d(\beta_n, id) = 2n$, by moving as necessary in $T_3$ and
compensating in either $T_1$ or $T_2$. We cannot apply Lemma \ref{hbeta:lem},
as the parameters of $\alpha_n$ grow along with $\beta_n$. But we can find
$d(\beta_n,\alpha_n)$ directly: we can move from $\beta_n$ to
$\alpha_n$ in $2n$ steps as well, by moving in $T_3$ as necessary, compensating
in $T_2$ for the first $n$ steps and then in $T_1$ for the final $n$ steps. So
$d(\beta_n,\alpha_n) = 2n = d(\beta_n,id)$. By Lemma \ref{seqstarinc:lem},
$\beta \in S(\alpha)$.
\end{proof}

Recall, the point $\zeta^i_k$, introduced in Definition \ref{zeta:def}, has $m_i
= l_i = k$, and all other parameters trivial.

\begin{lemma}
\label{boundm12frombeta:lem}
Let $z\in DL_3(q)$ have $m_i(z) > 0$ for $i=1$ or $2$. 
Then $d(z,\zeta^i_1) - d(z,id) \neq \beta(\zeta^i_1)$.
\end{lemma}
\begin{proof}
By Lemma \ref{nonpos-horo:lem}, $d(z,\zeta) \leq d(z,id)$, so $d(z,\zeta) -
d(z,id) \leq 0$, while Lemma \ref{pos-horo:lem} ensures $\beta(\zeta) > 0$.

%Since $m_1(\beta_n) = m_2(\beta_n) = 0$, this follows from $m$-invariance of
%horofunctions, Theorem \ref{minvariance:thm}.
\end{proof}

\begin{definition}
For $j \in \{1,2\}$, $k \in \integers_{\geq 0}$, and $\epsilon \in \{0,1,...,q-1\}$,
let $\nu^{j,\epsilon}_k$ be the point in $DL_3(q)$ with
$d(\nu^{j,\epsilon}_k,id) = k$ obtained by moving $k$ edges (all labeled
$\epsilon$) upward in $T_j$ and $k$ edges downward in $T_3$. So
$l_j(\nu^{j,\epsilon}_k) = m_3(\nu^{j,\epsilon}_k) = k$, while 
all other $m_i,l_i=0$.
\end{definition}

\begin{lemma} 
\label{boundT12frombeta:lem} 
Let $z \in DL_3(q)$ have $m_1(z) = m_2(z) = 0$, and assume 
$l_j(z) > 0$ for some $j \in \{1,2\}$.
Choose label $\epsilon$ not equal to the first edge joining $o_j$ to $p_j(z)$.  
Then $\nu = \nu^{j,\epsilon}_1$ has $d(z,\nu) - d(z,id) \neq \beta(\nu)$.  
\end{lemma}

\begin{proof}
By Lemma \ref{hbeta:lem}, $\beta(\nu) = -1$. Lemma
\ref{bigindex:obs} will ensure that $d(z,\nu) - d(z,id) \geq 0$ if for each
$\sigma \in \Sigma_d$, there is at least one $i$ such that $f_{\sigma,i}(z,\nu) \geq
d(z,id)$.

For each $k=1,2,3$, denote $m_k(z)$ and $l_k(z)$ by $m_k$ and $l_k$ respectively.
Assume for concreteness that $l_1 > 0$. Note that $m_3 = \sum_{k=1}^3 l_k$.
Also, $d(z,id) = d_3(o_3, p_3(z)) = l_1 + l_2 + 2l_3$ since Observation
\ref{treedist:obs} ensures this is a lower bound on $d(z,id)$, and it can be
realized by moving in $T_3$ and compensating in $T_1$ and $T_2$ as appropriate.

We have the following distances, by Lemma \ref{mixy:lem}.
\[ \begin{array}{ccc}
m_1(z, \nu) = l_1 \quad & m_2(z,\nu) = l_2 & m_3(z,\nu) = l_3  \\
l_1(z, \nu) = 1 \quad & l_2(z,\nu) = 0 & l_3(z,\nu) = \sum l_k - 1 \\
\end{array} \]

The following table provides, for each $\sigma$, an $f_{\sigma,i}(z,\nu) \geq l_1 + l_2 + 2l_3$:
\begin{center}
\begin{tabular}{|l|l|}
\hline
% (1)  & 3 & 
$f_{(1),3} = 2l_1 + l_2 + l_3 + (\sum l_k - 1)$  &
% (12) & 3 & 
$f_{(12),3} = 2l_2 + l_1  + l_3 + (\sum l_k - 1)$ \\ \hline
% (13) & 3 & 
$f_{(13),3} = 2l_3 + l_2 + l_1 + 1$              & 
% (23) & 2 & 
$f_{(23),2} = l_1 + l_3 + (\sum l_k - 1)$         \\ \hline
% (123)& 2 & 
$f_{(123),2} = l_2 + l_3 + (\sum l_k - 1) + 1$   &
% (132)& 3 & 
$f_{(132),3} = 2l_3 + l_1 + l_2$                  \\ \hline
\end{tabular}
\end{center}

% TODO Greg says....
%If $\sigma(2) = 3$, then $i=2$ works; if $\sigma(1) =3$ or $\sigma(3) = 3$,
%then $i=3$ works.  Is this better?

%%%\begin{center}
%%%\begin{tabular}{|c|c|l|l|}
%%%\hline
%%%$\sigma$ & $i$ & $f_{\sigma,i}$ & formula/notes \\ \hline
%%%%(1)  & 2 & $l_1 + l_2 + \sum l_i - 1$         & $2l_1 + 2l_2 + 1l_3 - 1$ \\
%%%     & 3 & $2l_1 + l_2 + l_3 + \sum l_i - 1$  & $3l_1 + 2l_2 + 2l_3 - 1$ \\ \hline
%%%%
%%%%(12) & 2 & $l_2 + l_1 + 1 + \sum l_i - 1$     & $2l_1 + 2l_2 + 1l_3$ \\
%%%     & 3 & $2l_2 + l_1  + l_3 + \sum l_i - 1$ & $2l_1 + 3l_2 + 2l_3 - 1$ \\ \hline
%%%%
%%%%(13) & 2 & $l_3 + l_2 + 1$                    & $0l_1 + 1l_1 + 1l_3 + 1$ \\
%%%     & 3 & $2l_3 + l_2 + l_1 + 1$             & $1l_1 + 1l_2 + 2l_3 + 1$ \\ \hline
%%%%
%%%(23) & 2 & $l_1 + l_3 + \sum l_i - 1$         & $2l_1 + 1l_2 + 2l_3 - 1$ \\
%%%%     & 3 & $2l_1 + l_3 + l_2$                 & $2l_1 + 1l_2 + 1l_3$ \\ \hline
%%%%
%%%(123)& 2 & $l_2 + l_3 + \sum l_i - 1 + 1$     & $1l_1 + 2l_2 + 2l_3$ \\
%%%%     & 3 & $2l_2 + l_3 + l_1 + 1$             & $1l_1 + 2l_2 + 1l_3 + 1$ \\ \hline
%%%%
%%%%(132)& 2 & $l_3 + l_1 + 1$                    & $1l_1 + 0l_2 + 1l_3 + 1$ \\
%%%     & 3 & $2l_3 + l_1 + l_2$                 & $1l_1 + 1l_2 + 2l_3$ \\ \hline
%%%\end{tabular}
%%%\end{center}

If $l_2 > 0$ instead, the calculation is symmetric.
\end{proof}

%%%\begin{itemize}
%%%\item[$\sigma=(1)$:] With $i=3$, $d(z,\nu) = d(z,id) + l_1 + (l_1 - 1)$
%%%\item[$\sigma=(12)$:] If $l_3 > 0$, $i=3$ is maximal and $d(z,\nu) =
%%%d(z,id) + 2l_2 + (l_1 - 1)$. Otherwise $l_3 = 0$; in this case, $i=3$ may
%%%still be maximal, with the same result, or $i=2$ is maximal and 
%%%$d(z,\nu) = 2d(z,id)$.
%%%\item[$\sigma=(13)$:] With $i=3$, $d(z,\nu) = d(z,id) + 1$
%%%\item[$\sigma=(23)$:] If $l_3 > 0$, $i=2$ is maximal and $d(z,\nu) =
%%%d(z,id) + (l_1 - 1)$. Otherwise, $l_3 = 0$, $i=3$ is maximal and
%%%$d(z,\nu) = d(z,id) + l_1$.
%%%\item[$\sigma=(123)$:] If $l_3 > 0$, $i=2$ is maximal and $d(z,\nu) =
%%%d(z,id) + l_2$. Otherwise, $l_3 = 0$, $i=3$ is maximal and $d(z,\nu) =
%%%d(z,id) + l_2 + 1$. 
%%%\item[$\sigma=(132)$:] If $l_3 > 0$, $i=3$ is maximal and $d(z,\nu) =
%%%d(z,id)$. If $l_2 > 0$, the same is true. Finally, if $l_2 = l_3 = 0$,
%%%then $i=2$ is maximal and $d(z,\nu) = d(z,id) + 1$.
%%%\end{itemize}
%%%
%%%We see that in all scenarios, $d(z,\nu) \geq d(z,id)$, whereby the
%%%$d(z,\nu) - d(z,id) \geq 0$. However, $\beta(z) = -1$ by Lemma
%%%\ref{hbeta:lem}.

\begin{definition}
For any finite set $F$, let
\[B(F) = \{ f \in \overline{DL_3(q)} \mid f|_{F} = \beta|_{F} \}.\]
%\[F = \{\zeta^1_1, \zeta^1_2 \} \cup \{\nu^{i,\epsilon}_1 \mid i \in
%\{1,2\}, \epsilon \in \{0,1\} \}\] 
and set
\[F_0 = \{\zeta^1_1,\ \zeta^1_2,\ \nu_1^{1,0},\ \nu_1^{1,1},\ \nu_1^{2,0},\
\nu_1^{2,1}\}\]

%TODO This needs to be cleaned up to say exactly what we want to say.

Because the functions in $\overline{DL_3(q)}$ map from the discrete space
$DL_3(q)$ to the discrete space $\integers$, when
we consider the compact-open topology on $\overline{DL_3(q)}$, the collection
$\{B(F)\}$, $F$ a finite set, is a neighborhood basis for $\beta$.
Thus $B(F_0)$ is an open set about $\beta$.
\end{definition}

\begin{theorem}
\label{onlybsbeta:thm}
A sequence $(b_n)\subset DL_3(q)$ approaches $\beta$ if and only if it has a
tail whose projections are trivial in $T_1$ and $T_2$ and $m_3(b_n) \rightarrow \infty$. 
\end{theorem}

\begin{proof}
By Lemmas \ref{boundm12frombeta:lem} and
\ref{boundT12frombeta:lem}, no $x \in DL_3(q)$ that is nontrivial in $T_1$ or $T_2$ lies
in $B(F_0)$. This proves the forward direction.

Now suppose $(b_n)$ is trivial in $T_1$ and $T_2$ and has $m_3(b_n)\rightarrow
\infty$. Then $l_3(b_n) = m_3(b_n)$. The calculations in the proof of Lemma
\ref{hbeta:lem} show that when $n$ is large relative to the parameters of $z
\in DL_3(q)$, the choice of upward path in $T_3$ is irrelevant to the
calculation. We can replace $n$ with $m_3(b_n)$ in that calculation, and we
obtain the same result.  \end{proof}

%%%\begin{lemma}
%%%\label{bn->beta:lem}
%%%Let $(b_n)$ be any sequence of points such that each $b_n$ is trivial in $T_1$
%%%and $T_2$, and $m_3(b_n) = n$. Then $(b_n)$ converges to $\beta$.
%%%\end{lemma}
%%%
%%%\begin{proof}
%%%Note that $l_3(b_n)$ must equal $n$.  
%%%\end{proof}
%%%if, in the calculation of
%%%$\beta$, replace $n$ with $m_3(b_n)$, we obtain the same result.
%%%
%%%Let $z \in DL_d(q)$, By taking $m_3(b_n)$ very large relative to the parameters
%%%of $z$, we the calculation
%%%of $d(\beta_n,z)$  shows that as $m_3(\beta_n)\rightarrow \infinity$ the path that
%%%$\beta_n$ selects upward in $T_3$ is irrelevant. The only difference between
%%%$(\beta_n)$ and $(b_n)$ is this choice of path.
%%%, this .  Any sequence $(b_n)$ approaching $h_\beta$ must have
%%%$m_3(b_n) \rightarrow \infty$, and by Lemma \ref{bn->beta:lem}, the actual
%%%choice of upward path in $T_3$ becomes irrelev
%%%\end{proof}

\begin{corollary}
The horofunction $\beta$ is isolated in $\partial DL_3(q)$. 
\end{corollary}

\begin{proof}
Let $(h_n) \subset \partial DL_3(q)$ be a sequence of horofunctions approaching
$\beta$. Then $(h_n)$ has a tail $(t_n)$ lying in $B(F_0)$. For each $n$, any
sequence $(x_k) \in DL_3(q)$ approaching $t_n$ also has a tail in $B(F_0)$. By
Theorem \ref{onlybsbeta:thm}. $(t_n)$ is the constant sequence $(\beta)$.
\end{proof}
%%%By Theorem \ref{onlybsbeta:thm}, any point in $DL_3(q)$
%%%having nontrivial $T_1$ or $T_2$ must lie outside the open set $B(F)$ containing
%%%$h_\beta$. Similarly, any horofunction $h'$ representable by a sequence
%%%containing infinitely many such points lies outside of $B(F)$. However, a
%%%sequence $(b_n)$ eventually containing no such points must have $m_3(b_n)
%%%\rightarrow \infty$ to define a horofunction, by Observation \ref{minfnonempty:obs}. 
%%%By Theorem \ref{onlybsbeta:thm}, the resulting horofunction must be $h_\beta$.

It is worth noting that while $\beta$ is isolated in the  boundary, $S(\beta)
\neq \{\beta\}$. For example, let $Y\subseteq \{1,2,3\}$ contain 3 and at least
one other index.  Let $(\gamma_n)$ be any sequence with $m_i(\gamma_n) =
l_i(\gamma_n) = n$ for $i \in Y$, and $m_i(\gamma_n) = l_i(\gamma,_n) = 0$
otherwise. It can be shown that $(\gamma_n)$ defines a horofunction $\gamma
\neq \beta$; and one can calculate $d(\gamma_n,id) \geq d(\gamma_n,\beta_n)$ for
each $n$, so that Lemma \ref{seqstarinc:lem} ensures $h_\gamma \in S(\beta)$.

\begin{corollary}
For $k \geq 0$, let 
\[N_k(\beta) = \{ z\in DL_3(q)\,|\, m_i(z) = l_i(z) = 0 \text{ for } i=1,2 \text{ and } m_3(z) =
l_3(z) \geq k\} \cup \{\beta\}.\] 
Then the family $\{N_k(\beta)\}$, $k \geq 0$, is a neighborhood basis for
$\beta \in \overline{DL_3(q)}$.
\end{corollary}

Note that $\beta_n$ and $\beta$ are not the only elements of
$N_k(\beta)$, since
alternate choices of edge labels may be made.

\begin{theorem}
Let $h \in \partial DL_3(q)$ have $m_3(h) = 0$. Then $h\not \in S(\beta)$.
As a consequence, $\alpha \not \in S(\beta)$.
\end{theorem}
\begin{proof}
%By definition, 
%\[ S(\beta) = \overline{\bigcup_{C\geq0}\left(\bigcap_{k>0}
%\overline{H(N_k(\beta),C)}\right)}.\] 

By $m$-invariance, any sequence of points in $DL_3(q)$ approaching $h$ has a
tail $(a_n)$ such that for each $n$, $m_3(a_n) = 0$. 
%satisfying the conditions of Lemma \ref{m1m30:lem}.
For any $C\geq 0$, when $k > C$,  the structure of $N_k(\beta)$ and Lemma
\ref{pos-horo:lem} together imply that 
\[d(a_n,N_k(\beta)) \geq d(a_n,id) + k > d(a_n,id) + C.\] 
So  $a_n \not \in H(N_k(\beta),C)$, and so $h \not \in
\cap_{k>0} \overline{H(N_k(\beta),C)}$.

By Corollary \ref{mconv:cor}, any sequence $(h_n)$ of horofunctions approaching
$h$ must also eventually have $m_3(h_n) =0$ as well, so $h$ is also not a limit
of any sequence of horofunctions in the union of $\bigcap_{k>0}
\overline{H(N_k(\beta),C))}$ over all $C$.

%Again, by $m$-invariance, any sequence $(h_n)$ of horofunctions approaching
%$h_\alpha$ must also have a tail $(g_n)$ with trivial $m_3$, and as with
%$h_\alpha$, each $g_n$ must be represented by a sequence of points $(x_n)$
%eventually satisfying  $d(x_n,N_k(\beta) \geq d(x_n,id) + k)$, and so each $g_n$
%is similarly not in $\overline{H(N_k(\beta),C)}$ for any $C$.

Thus 
\[ h \not \in \overline{\bigcup_{C\geq0}\left(\bigcap_{k>0}
\overline{H(N_k(\beta),C)}\right)} =
S(\beta).\] 

\end{proof}

\bibliographystyle{plain}
\bibliography{non-symmetric-star}

\end{document}